\documentclass{amsart}
\usepackage{mathrsfs}
\usepackage{amsmath,amsthm,amssymb,microtype,color,verbatim,hyperref,multicol,tikz}
\usepackage{stmaryrd}
\makeatletter
\def\imod#1{\allowbreak\mkern10mu({\operator@font mod}\,\,#1)}
\makeatother

\newcommand{\gap}{\texttt{GAP}}

 \newcommand{\Solv}{\text{Sol}}
  \newcommand{\Sol}{\text{Sol}}
  \newcommand{\ws}{\mu_s}
\newtheorem{theorem}{Theorem}[section]
\newtheorem{lemma}[theorem]{Lemma}

\newtheorem{corollary}[theorem]{Corollary}
\newtheorem{conjecture}[theorem]{Conjecture}

\theoremstyle{definition}
\newtheorem{definition}[theorem]{Definition}
\theoremstyle{remark}
\newtheorem{remark}[theorem]{Remark}

\newtheorem{question}[theorem]{Question}

\title[Solvablizer  Numbers]{Solvablizer Numbers of Finite  Groups}

\author[B.~Akbari]{Banafsheh~Akbari}
\address{Department of Mathematics\\
Cornell University\\
Ithaca, NY 14853 \\
USA}
\email{b.akbari@cornell.edu}

\author[T.~Foguel]{Tuval~Foguel}
\address{Department of Mathematics and Computer Science \\
Adelphi University \\
Garden City, NY 11010\\ USA}
\email{tfoguel@adelphi.edu}

\author[J.~Schmidt]{Jack~Schmidt}
\address{Department of Mathematics\\
University of Kentucky\\
Lexington, KY 40506 \\
USA}
\email{jack.schmidt@uky.edu}

\begin{document}

%\maketitle

\begin{abstract}
Consider a nonsolvable finite group $G$, where $R(G)$ represents the solvable radical of $G$. For any $x \in G$, the solvabilizer of $x$ in $G$ is defined as $\Solv_G(x) = \{ y \in G \mid \langle x, y\rangle\ \text{is solvable}\}$. Notably, the entirety of $G$ can be expressed as the union over all $x \in G\setminus R(G)$ of their respective solvabilizers: $G = \cup_{x\in G\setminus R(G)}\Solv_G(x)$. A solvabilizer covering of $G$ is characterized by a set $X\subset G\setminus R(G)$ such that $G = \cup_{x\in X}\Solv_G(x)$. The solvabilizer number of $G$ is then defined as the minimum cardinality among all solvabilizer coverings of $G$. This paper delves into the exploration of the solvabilizer number for diverse nonsolvable finite groups $G$, shedding light on the interplay between solvability and the structure of these groups.
\end{abstract}

\maketitle
%________________________
\section{Introduction}
%________________________
This paper is solely focused on finite groups. In 1968, John G. Thompson demonstrated in [\cite{Thompson}, Corollary 2] that a finite group $G$ is soluble if and only if every two-generated subgroup of $G$ is soluble. This discovery motivates the introduction of the solvabilizer of $x$ in $G$, denoted as $$\Solv_G(x) = \{ y\in G\mid \langle x, y\rangle\ \text{is solvable}\}.$$ 
In 2006 \cite{GKPS}, Guralnick et al. extended Thompson's theorem, proving that the solvable radical of a finite group $G$ coincides with the set of elements $x\in G$ for which $\Solv_G(x)=G$. This insight underscores the significance of the solvabilizer, which operates similarly to the centralizer in a group, while the solvable radical $R(G)$ plays a role akin to the center of a group. Analogous to how the center corresponds to the set of elements $x\in G$ for which $C_G(x)=G$, the centralizer of $x\in G$ can be defined analogously as $C_G(x)=\{ y\in G\mid \left\langle x, y\right\rangle\ \text{is abelian}\}$. It's crucial to note that $\Solv_G(x)$, unlike the centralizer, may not necessarily form a subgroup. 

A group is considered covered by a collection of subsets if it can be expressed as the union of that collection. It's clear that if $G$ is nonsolvable, then $G = \cup_{x\in G\setminus R(G)}\Solv_G(x)$. A solvabilizer covering of $G$ is a set $X\subset G\setminus R(G)$ such that $G = \cup_{x\in X}\Solv_G(x)$. Similar to the covering number of a group \cite{GKS}, we introduce the solvabilizer number of $G$, defined as the minimum cardinality among all solvabilizer coverings of $G$. This paper delves into the investigation of the solvabilizer number for various nonsolvable finite groups $G$.

In a parallel manner to the covering number, we establish that for any nonsolvable group, the solvabilizer number is $\ge 3$, mirroring the covering number of a group also being $\ge 3$. Moreover, we demonstrate that the solvabilizer number of $A_5$ is three. Our exploration extends to investigating the solvabilizer number of Cartesian products and scrutinizing the solvabilizer number for various families of simple groups. Additionally, we leverage \cite{GAP4} to determine and/or establish bounds on the solvabilizer number for numerous smaller nonsolvable groups.
%________________________
\section{Theorems}
%________________________
In this section, we will articulate and demonstrate several theorems concerning the solvabilizer number.
\begin{theorem}\label{quotient}
Let $G$ be a nonsolvable group such that $R(G)\neq 1$. The set $X\subset G \setminus R(G)$ is a solvablizer covering of $G$ if and only if the set $XR(G)\subset G/R(G)$ consisting of all cosets $xR(G)$, $x\in X$, is a solvablizer covering of $G/R(G)$.
\end{theorem}
\begin{proof}
Use the following as a property of the solvabilizer sets {\rm (\cite[Lemma 2.5]{ALMM})} to get the desired result.
$$\Sol_{G/R(G)} (xR(G)) = \Sol_G (x)/R(G).$$
\end{proof}
Using Theorem \ref{quotient}, we can reduce the solvablizer covering problem to Fitting-free groups, for example, direct and wreath products of almost simple groups.

\begin{remark}\label{solofinvolution} If $x,y\in G$ are involutions, then  $\left\langle x, y\right\rangle\cong D_{2n}$  where $n$ is the order of $xy$ and thus is solvable.
\end{remark}
\begin{theorem}
If $G$ is a nonabelian simple group,  $x\in G$  is an involution, then $\Solv_G(x)$ is not a subgroup.\end{theorem}
\begin{proof}
Since $\Solv_G(x)$ contains all involutions in $G$ and $G$ is nonabelian simple group if $\Solv_G(x)$ was a subgroup it will equal $G$ and force $x\in R(G)$.
\end{proof}

\begin{definition} A solvablizer covering of $G$ is a set $X\subset G \setminus R(G)$, such that $G=\cup_{x\in X}\Solv_G(x)$, and the solvablizer number of $G$ is the minimum cardinality of all solvablizer covering of $G$.
	A minimal solvablizer covering of $G$ is a solvablizer covering $X\subset G\setminus R(G)$, such that $\mid X\mid=$  the solvablizer number of $G$. For convenience, we denote the solvabilizer number of $G$ by $\alpha(G)$.
\end{definition} 
\begin{lemma} If $x\in G$ a finite group, then for any $n$ we see that $\Solv_G(x)\subseteq \Solv_G(x^n)$.\end{lemma}
\begin{proof} Note that if $\left\langle x, y\right\rangle$ is solvable, then so is $\left\langle x^n, y\right\rangle$.\end{proof}

\begin{corollary}\label{primes} If $G$ is a finite nonabelian simple group, then it has a minimal solvablizer covering composed of elements of prime orders.\end{corollary}

\begin{question}
Is there a nonsolvable group $G$ with $\alpha(G)=2$?
\end{question}
We can use the following theorem to answer this question.
\begin{theorem} [\cite{GKPS}]\label{GKPS-nonsolvable}
Let $G$ be a finite group. Suppose that $x$ and $y$ are not in $R(G)$. Then there
exists an element $s\in G$ such that $\langle x, s\rangle$ and $\langle y, s\rangle$ are not solvable.
\end{theorem}
\begin{corollary}\label{alphag2}
Given a finite nonsolvable group $G$, $\alpha(G)>2$.
\end{corollary}
\begin{proof} 
In the case where $R(G)=1$, this is a straightforward result of Theorem \ref{GKPS-nonsolvable}. We can use Theorem \ref{quotient} to show the similar result when $R(G)\neq 1$.
\end{proof}

By using \cite{GAP4},  we found three involutions whose solvablizers cover $A_5$.  So we can state the following theorem.   
\begin{theorem}
 The solvablizer number of $A_5$ is $3$.
\end{theorem}
\begin{remark}
    It is worth noting that any covering of $A_5$ by three solvablizers  will be a covering by three solvablizers of involutions.
\end{remark}
\begin{definition}
Given a nonsolvable group $G$, we will say that elements $x,y\in G$ are non-Solvablized if $\left\langle x, y\right\rangle$ is  not a solvable group.
The cardinality of a maximal subset of pairwise non-Solvablized elements of $G$ is denoted by $\ws(G)$.
\end{definition}
\begin{theorem}
 Given a nonsolvable finite group $G$, $\alpha(G)\le \ws(G)$.
\end{theorem}
\begin{proof}
Let $X$ be a set of a maximal subset of pairwise non-Solvablized elements of $G$. Given $y\in G\setminus X$ there is an $x\in X$ with  $\left\langle x, y\right\rangle$ is a solvable group by the maximality of $X$. Therefore $G=\cup_{x\in X}\Solv_G(x)$.
\end{proof}
\begin{definition}
    If $G$ is a two generated noncyclic finite group, then we denote by $\mu(G)$ the maximal order subset of pairwise generators of $G$ \cite{Britnell}.
\end{definition}
\begin{remark}
    If $G$ is a minimal simple group, the $\ws(G)=\mu(G)$. So for a minimal simple group, $\alpha(G)\le \mu(G)$.
\end{remark}
\begin{remark}
 Note that the  set  below is a maximal subset of pairwise generators of $A_5$ which is a minimal simple group.
$$ \{(1\ 2\ 3), (3\ 4\ 5), (1\ 2\ 3\ 4\ 5), (1\ 2\ 3\ 5\ 4),(1\ 2\ 4\ 3\ 5),
(1\ 2\ 4\ 5\ 3), (1\ 2\ 5\ 3\ 4), (1\ 2\ 5\ 4\ 3)\}.$$
And $\alpha(A_5)=3\le \mu(A_5)=8$.
\end{remark}

\begin{definition}
    The covering number of a finite group $G$ is $\sigma(G)$ is the least integer $k$ such that $G$ is the union of $k$ proper subgroups .
\end{definition}
Clearly, given a two generated noncyclic finite group $G$, $\mu(G)\le \sigma(G)$.

\begin{corollary}
    If $G$ is a minimal simple group, then $\alpha(G)\le \sigma(G)$.
\end{corollary}
\begin{remark}
 Note that the  $\sigma(A_5)=10>\mu(A_5)=8>\alpha(A_5)= 3$ \cite{GKS}.
\end{remark}

Now the following questions arises naturally.
\begin{question}
\

Are there examples of  nonsolvable groups, with $\ws(G)$ coincides with the solvablizer number?

Are there examples of  minimal simple groups, with $\mu(G)$ coincides with the solvablizer number?

And is it true that for any nonsolvable group $G$,  $\alpha(G)\le \sigma(G)$? 
\end{question}

%________________________
\section{Products}
%________________________

In this section, we will formulate and demonstrate several theorems related to the solvabilizer number of Cartesian products.
\newcommand\PSL[2]{\operatorname{PSL}_{#1}\left(#2\right)}
\newcommand\GL[2]{\operatorname{GL}_{#1}\left(#2\right)}
\newcommand\GammaL[2]{\operatorname{\Gamma{}L}_{#1}\left(#2\right)}
\newcommand\GF[1]{\mathbb{F}_{#1}}
\begin{lemma}
    If $G$ is a nonsolvable finite group and $G=H\times K$, then $\Solv_G((x,y))=\Solv_H(x)\times \Solv_K(y)$.
\end{lemma}
\begin{proof}
    Note that the Cartesian product of two solvable groups is solvable, thus $\Solv_H(x)\times \Solv_K(y)\subseteq\Solv_G((x,y))$. Now if $\Solv_H(x)\times \Solv_K(y)\neq\Solv_G((x,y))$, we can find $(h,k)\in \Solv_G((x,y))\setminus\Solv_H(x)\times \Solv_K(y) $, but then $\left\langle x, h\right\rangle$ and $\left\langle y, k\right\rangle$ are solvable a contradiction.
\end{proof}
\begin{theorem}\label{Car}
    Let $G$ be a nonsolvable finite group and $G=H\times K$.
    \begin{enumerate}
        \item If $K$ is a solvable group, then $\alpha(G)=\alpha(H)$.
         \item If $H$ and $K$ are nonsolvable groups, then $\alpha(G)=\min\{\alpha(H),\alpha(K)\}$.
         
    \end{enumerate}
    
\end{theorem}
\begin{proof} 
    Without loss of generality assume $K$ is a solvable group or that $\alpha(K)\ge\alpha(H)$. Let $X=\{ x_1,\dots x_n\}\subseteq H$ a minimal solvablizer covering of $H$. Let  $\Hat{X}=\{ (x_1,1),\dots (x_n,1)\}\subseteq G$, Note that the $\Solv_G((x_i,1))=\Solv_H(x_i)\times K$. Therefore, $$G=\bigcup_{(x_i,1)\in\Hat{X}}\Solv_G((x_i,1)).$$ Hence $\alpha(G)\le\alpha(H)$.
    To prove (2) we will need two cases.
    
    {\bf Case 1:}
    
    Assume that $\alpha(G)<\alpha(H)$ and $K$ is a solvable group. Let $$\Bar{X}=\{ (x_1,y_1),\dots (x_m,y_m)\}\subseteq G $$ be a minimal solvablizer covering of $G$ and $\pi_1$ the projection homomorphism onto $H$ and and $\pi_2$ the projection homomorphism onto $K$. 

       Given that $K$ is a solvable group, we then see that 
    $$H=\bigcup_{x_i\in\pi_1(\Bar{X})}\Solv_H(x_i).$$ A contradiction.

 {\bf Case 2:}
    
    Assume that $H$ and $K$ are nonsolvable and that $\alpha(G)<\alpha(H)\le \alpha(K)$. Let $$\Bar{X}=\{ (x_1,y_1),\dots (x_m,y_m)\}\subseteq G $$ be a minimal solvablizer covering of $G$. 

       We then see that 
    $$H=\bigcup_{x_i\in\pi_1(\Bar{X})}\Solv_H(x_i).$$ To avoid a contradiction  at least one of the $x_i\in R(H)$; note that this implies that $y_i\notin R(K)$.
 Let $$T_H=\pi_1(\Bar{X})\setminus R(H)\text{\ and\ }T_K=\pi_2(\Bar{X})\setminus R(K),$$ Note if one of them is the empty set then the same argument as in case 1 works. Thus we can assume that both are not empty. Let $x\in H\setminus T_H$ and $y\in K\setminus T_K$, clearly $$(x,y)\notin \bigcup_{(x_i,y_i)\in\Bar{X}}\Solv_G((x_i,y_i)).$$ A contradiction.
    \end{proof}

 \begin{definition} An involutionary solvablizer covering of $G$ is a set of involutions $X\subset G \setminus R(G)$, such that $G=\cup_{x\in X}\Solv_G(x)$, and the involutionary solvablizer number of $G$ is the minimum cardinality of all involutionary solvablizer covering of $G$.
	A minimal involutionary solvablizer covering of $G$ is an involutionary solvablizer covering $X\subset G\setminus R(G)$, such that $\mid X\mid=$  the involutionary solvablizer number of $G$. For convenience, we denote the involutionary solvabilizer number of $G$ by $\alpha_{inv}(G)$.
\end{definition} 
\begin{remark}
    If $G$ is a nonsolvable group with no involutionary solvabilizer covering we will say that $\alpha_{inv}(G)=\infty$. From the table in Section~\ref{Appendix} we see that $\PSL{2}{7}$ is the group of smallest order with $\alpha_{inv}(G)=\infty$.
\end{remark}

\begin{theorem}\label{Inv}
    Let $G$ be a nonsolvable finite group and $G=H\times K$ and $H$ has an involutionary solvabilizer covering.
    \begin{enumerate}
        \item If $K$ is a solvable group, then $\alpha_{inv}(G)=\alpha_{inv}(H)$.
         \item If $H$ and $K$ are nonsolvable groups, then $\alpha_{inv}(G)=\min\{\alpha_{inv}(H),\alpha_{inv}(K)\}$.
         
    \end{enumerate}
    
\end{theorem}
\begin{proof}
The proof is basically identical to the proof of Theorem~\ref{Car}, just replace the $(x_i, y_i)$s  by involutions and note that for each $i$ that  one of the $x_i$ or $y_i$ is an involution and the other is either an involution or the identity.
\end{proof}

\begin{remark} If $G=H\times K$ where $\infty>\alpha_{inv}(H)>\alpha(K)$ and $\alpha_{inv}(K)=\infty$, then $\alpha(G)<\alpha_{inv}(G)<\infty$. For example when $G=\PSL{2}{7}\times \PSL{2}{9}$ we see that $\alpha(G)=5<\alpha_{inv}(G)=9$. This agrees with \cite{GAP4}.
\end{remark}

\begin{theorem}\label{TW}
    If $G=H\wr K\le H\wr S_n$ where $H$ is a nonsolvable group and $K\le S_n$ acting on $B=H^n$, then $\alpha(G)\le\alpha(H)$.\end{theorem}
    \begin{proof}
        Let $X=\{x_1,\dots x_j\}\subseteq H$ be a minimal solvablizer covering of $H$ and let $T_i=\Solv_H(x_i)^n\subseteq B$. Now let
        $(T_i,k)=\{(y,k)\ \mid\ y\in T_i\}$ for any $k\in K$, and $$\bar{T_i}=\cup_{k\in K} (T_i,k).$$
        Note that $\bar{T_i}\subseteq\Solv_G((x_i,\dots, x_i,1_K))$ and that $\cup_{1\le i\le j}\bar{T_i}=G$. Therefore $\alpha(G)\le \alpha(H)$
    \end{proof}
\begin{remark}
    Note that according to Theorem~\ref{Car} $\alpha(B=H^n)=\alpha(H)$.
    Thus, we can rephrase Theorem~\ref{TW} as follows:

    If $G=H\wr K\le H\wr S_n$ where $H$ is a nonsolvable group and $K\le S_n$ acting on $B=H^n$, then $\alpha(G)\le\alpha(B)$.
    \end{remark}
\begin{corollary}
    If $G=A_5\wr K$ where $K\le S_n$ acts on $B=A_5^n$ as a $n$-cycle, then $\alpha(G)=3$.
\end{corollary}   
\begin{proof}
    By Theorem~\ref{TW}, $\alpha(G)\le\alpha(A_5)=3$ and by Corollary~\ref{alphag2} $\alpha(G)>2$. Thus  $\alpha(G)=3$.
\end{proof}

%\begin{theorem}\label{TW2}
   % If $G=H\wr K\le H\wr S_n$ where $K$ is a nonsolvable group and $K\le S_n$ acting on $B=H^n$, then $\alpha(G)\le\alpha(K)$.\end{theorem}
   % \begin{proof}
    %   Let $Y=\{y_1,\dots, y_j\}\subseteq K$ represent a minimal solvabilizer covering of $K$, and let $T_i=\Solv_G((1_B,y_i))$. It is noteworthy that for any $b\in B$ and $k\in \Solv_K(y_i)$, we observe that $(b,k)\in \Solv_G((1_B,y_i))$. Consequently, $\cup_{1\le i\le j}T_i=G$. Hence, $\alpha(G)\le \alpha(K)$.
   % \end{proof}
%\begin{remark}
  %  In the following corollary, we observe that wreath products exhibit a similar behavior to direct products concerning the Solvabilizer number.
%\end{remark}
  %  \begin{corollary}\label{CW}
  %  If $G=H\wr K\le H\wr S_n$ where $K$ and $H$ are nonsolvable groups and $K\le S_n$ acting on $B=H^n$, then $\alpha(G)\le\min{(\alpha(H),\alpha(K))}$. \end{corollary}
%________________________    
\section{Minimal simple groups and PSL} \label{minimal simple}
%________________________

In this section, we will state and provide proofs for several theorems regarding the solvabilizer number of all minimal simple groups and certain projective special linear groups.

\begin{remark}\label{Solminimalsimple}
	If $G$ is a minimal simple group,  $x\in G\setminus\{1\}$  and $\mathscr{M}_x=\{ M<G\mid x\in M\ \text{and}\  M \ \text{is maximal in}\ G\}$, then $\Solv_G(x)=\cup_{M\in\mathscr{M}_x}M$
\end{remark}

Thompson \cite[Corollary 1]{Thompson} determined all minimal simple groups. Every minimal simple group is isomorphic to one of the following: 
\begin{itemize} 
\item[{\rm (a)}]  $\PSL{2}{2^p}$ where $p$ is any prime; 
\item[{\rm (b)}]  $\PSL{2}{3^p}$ where $p$ is an odd prime; 
\item[{\rm (c)}]  $\PSL{2}{p}$ where $p > 3$ is a prime satisfying $p \equiv 2,3 \mod 5$; 
\item[{\rm (d)}]  ${\rm Sz}(2^p)$, where $p$ is an odd prime; and \item[{\rm (e)}] $\PSL{3}{3}$.
\end{itemize}
\begin{remark}
By utilizing Dickson's classification of maximal subgroups in $\PSL{2}{q}$ (\cite{Dickson}) and referring to Suzuki's paper (\cite{Suzuki}), we can discern all distinct maximal subgroups of $G$ while excluding $\PSL{3}{3}$. Moreover, the maximal subgroups of $\PSL{3}{3}$ can be identified using \cite{GAP4}.

Note that if  $G$ is one of the minimal simple groups $\PSL{2}{2^p}$ and ${\rm Sz}(2^p)$,then all maximal subgroups of $G$ have even order. Let $\mathscr{X}=\{\text{all involutions in } G\}$  for these cases, and we observe that $G=\cup_{x\in\mathscr{X}}\Solv_G(x)$.

If $G$ is one of the minimal simple groups $\PSL{2}{3^p}$, $\PSL{2}{p}$, or $\PSL{3}{3}$, then $G=\cup_{x\in\mathscr{X}\cup \mathscr{Y}}\Solv_G(x)$ where $\mathscr{Y}$ is the set of all elements of order $r$ in $G$, with $r=3$ in $G=\PSL{2}{3^p}$ or $\PSL{3}{3}$, and $r=p$ in $G=\PSL{2}{p}$.
\end{remark}
\begin{remark}
In the work by the authors in \cite{LM}, the following results have been established:
\begin{itemize}
    \item  If $G=L_2(q)$ and $q>9$ is a power of an odd prime, then $$\mu(G)=\sigma(G)=q(q+1)/2+1,$$ 
    \item if  $G=L_2(q)$ and $q\geq 4$ is an even prime power, then $$\mu(G)=[(q^2+1)/2]<\sigma(G)=q(q+1)/2.$$
    
    \item They also showed that for $G={\rm Sz}(q)$, $$\mu(G)=q^4/2<\sigma(G)=q^2(q^2+1)/2.$$ 
    \end{itemize}
    Leveraging these established findings, we can derive an upper bound for the solvabilizer number of minimal simple groups.
\end{remark}
Now, we inquire about the accuracy of these upper bounds. Specifically, is there a minimal simple group $G$ for which $\alpha(G)=\mu(G)$?

\begin{theorem}\label{PSL(2, 2^f)}
The minimal simple group $\PSL{2}{q}$, where $q=2^p$ with $p$ as a prime, can be covered by the solvabilizers of $q-1$ involutions. In particular, $\alpha_{inv}(\PSL{2}{q})=\alpha(\PSL{2}{q})=q-1$.
\end{theorem}
\begin{proof}
Firstly, we observe that, according to the well-known theorem by Dickson (\cite{Huppert}), all elements of $G=\PSL{2}{2^p}$ are of order either $2$, a factor of $2^p-1$ or $2^p+1$.

Utilizing Dickson's result (\cite{Dickson}), we can determine that every maximal subgroup of $\PSL{2}{2^p}$ is isomorphic to one of the following: $D_{2(q-1)}$, $D_{2(q+1)}$, and ${C_2}^p\rtimes C_{q-1}$, where $q=2^p$.

It is evident from \cite[Theorem 2.1]{King} that there exists a single class of $q+1$ conjugate abelian groups of order $q$, a single class of $q(q+1)/2$ dihedral groups $D_{2(q-1)}$, and a single class of $q(q-1)/2$ dihedral groups $D_{2(q+1)}$.
We note that $G$ possesses $q+1$ conjugate maximal subgroups of the form ${C_2}^p\rtimes C_{q-1}$, which contain all abelian Sylow $2$-subgroups of $G$. As all Sylow $2$-subgroups of $G$ intersect trivially and each of them contains $q-1$ involutions, we can find that $G$ has $q^2-1$ involutions belonging to the subgroups ${C_2}^p\rtimes C_{q-1}$. 

Note that all the $q^2-1$ involutions  are conjugate in $G$. If $x$ is an involution in a maximal subgroup $M\le G$, then the number of subgroups isomorphic to $M$ that contain $x$ is the number of conjugates of $x$ in $M$ multiplied by the number of conjugates of $M$ in $G$ divided by $q^2-1$. 

We can use the fact above and the fact that each maximal subgroup $D_{2(q+1)}$  (respectively, $D_{2(q-1)}$) has $q+1$ (respectively, $q-1$) involutions to show that every involution belongs to $\frac{(q+1){q(q-1)/2}}{(q^2-1)}=q/2$ dihedral groups $D_{2(q+1)}$ and $\frac{(q-1) {q(q+1)/2}}{(q^2-1)}=q/2$ dihedral groups $D_{2(q-1)}$ and one ${C_2}^p\rtimes C_{q-1}$. 

It is seen from \cite[Theorem 2.1]{King} that there is a single class of $q(q-1)/2$ (respectively, $q(q+1)/2$) conjugate cyclic groups of order $d$ for each divisor $d$ of $q+1$ (respectively, $q-1$). So it is easily seen that every two maximal dihedral groups $D_{2(q+1)}$ (respectively, $D_{2(q-1)}$) intersect at one involution if they intersect non trivially. If any two maximal subgroups ${C_2}^p\rtimes C_{q-1}$ intersect non trivially, then their intersection will be $C_{q-1}$ since any involution belongs to exactly one such maximal subgroup. 

We will show that the set of all maximal dihedral groups $D_{2(q+1)}$ can be partitioned into $q-1$ disjoint subsets which each of them contains $q/2$ maximal dihedral groups $D_{2(q+1)}$. Given an involution $x$, we pick all $q/2$ maximal subgroups $D_{2(q+1)}$ containing $x$ as the first subset of the partition whose members intersect at $x$. Now, let $y$  be an involution distinct from $x$. We can take the set of all $q/2$ maximal dihedral groups $D_{2(q+1)}$ containing $y$ as the second subset of the partition. We continue this process to find all $q-1$ disjoint subsets of the set of all maximal dihedral groups $D_{2(q+1)}$ to partition this set. These $q-1$ involutions cover all maximal dihedral groups $D_{2(q+1)}$. By a similar argument, it can be shown that the subset of all maximal dihedral groups $D_{2(q-1)}$ which contains $q(q-1)/2$ maximal groups $D_{2(q-1)}$ can be patitioned based on these $q-1$ involutions. So only $q$ maximal dihedral groups $D_{2(q-1)}$ are left to be covered that can be partitioned into two subsets which each contains $q/2$ maximal dihedral group $D_{2(q-1)}$. On the other hand, these $q-1$ involutions can cover $q-1$ maximal subgroups ${C_2}^p\rtimes C_{q-1}$ as well. Thus two maximal subgroups ${C_2}^p\rtimes C_{q-1}$ are left to be covered. 

In what follows, we demonstrate that the remaining maximal subgroups will be also covered by the $q-1$ involutions that have been already chosen. Note that these two maximal subgroups ${C_2}^p\rtimes C_{q-1}$ and $q$ maximal dihedral groups $D_{2(q-1)}$ only contain the involutions and elements of order $d$ as divisors of $q-1$. We can use Remark \ref{solofinvolution} to show that the involutions can be covered by the solvabilizer of any involution of $G$. So we only need to cover the order $d$ as divisors of $q-1$ elements of these maximal subgroups.

On the other hand, if a maximal subgroup $D_{2(q-1)}$ and ${C_2}^p\rtimes C_{q-1}$ intersect non trivially,  their intersection can contain the involutions and a single cyclic subgroup $C_{q-1}$. 
Considering \cite[Theorem 2.1]{King}, there are $q(q+1)/2$ cyclic subgroup $C_{q-1}$ which each of them is contained in exactly one maximal subgroup $D_{2(q-1)}$ and $\frac{q(q+1)}{q(q+1)/2}=2$  maximal subgroups ${C_2}^p\rtimes C_{q-1}$. Each of the two maximal subgroups ${C_2}^p\rtimes C_{q-1}$ must intersects one of the maximal subgroups $D_{2(q-1)}$ in the partition we found in the previous step and this intersection contains a cyclic subgroup $C_{q-1}$ since otherwise, the set consisting of the two maximal subgroups ${C_2}^p\rtimes C_{q-1}$ and $q$ maximal dihedral groups $D_{2(q-1)}$ is disjoint with the set of $q-1$ maximal subgroups ${C_2}^p\rtimes C_{q-1}$ and $q(q-1)/2$ dihedral groups $D_{2(q-1)}$ that we constructed in the previous step which is impossible.

Finally, we assert that the minimal solvabilizer covering of $G$ comprises these $q-1$ involutions. We have demonstrated that to cover the maximal subgroups $D_{2(q+1)}$, $q-1$ involutions are necessary. Notably, all maximal subgroups $D_{2(q+1)}$ can be only covered by the solvabilizers of involutions and/or solvabilizers of elements of order divisors of $q+1$, which, by Corollary ~\ref{primes}, can be assumed to be prime. In fact, considering an element $x$ of order $d$ as a prime divisor of $q+1$, we can see from Remark \ref{Solminimalsimple} that
$\Solv_G(x)=D_{2(q+1)}$ as the cyclic subgroup generated by $x$ is contained in precisely one maximal subgroup $D_{2(q+1)}$. Hence, at least $q(q-1)/2$ elements of a prime divisor of $q+1$ order are required to cover all maximal dihedral groups $D_{2(q+1)}$. This concludes the proof.
\end{proof}
\begin{theorem}\label{PSL(2, q)}
Let $G=\PSL{2}{p}$ be a minimal simple group, where $p>3$ is a prime satisfying $p \equiv 2,3 \mod 5$. Then $$\alpha(G)\geq\left \{\begin{array}{lll}
p, &  & p \equiv 1 \mod 4 \\[0.1cm]
(3p-1)/2, & & p \equiv 3 \mod 4 \end{array} \right. $$
\end{theorem}
\begin{proof}
According to \cite[Theorem 2.2]{Giudici}, the maximal subgroups of $G$, up to isomorphism, include: $D_{p-1}$, $D_{p+1}$, ${C_p}\rtimes C_{(p-1)/2}$, and either $A_4$ or $S_4$. We will consider the cases where $(p-1)/2$ is divisible by $2$ or not separately, resulting in two possibilities: when $p \equiv 1 \mod 4$ and $p \equiv 3 \mod 4$.

First let $p \equiv 1 \mod 4$. As we can see from \cite[Theorem 2.1]{King}, $G$ has $p(p+1)/2$ subgroup of order $2$ which follows there are $p(p+1)/2$ involutions in $G$. In the subsequent analysis, we demonstrate that to cover all maximal subgroups $D_{p+1}$, $p$ involutions are required. Considering \cite[Theorem 2.1]{King} there exist a single class of $p(p-1)/2$ conjugate cyclic groups of order $d$ for each divisor $d$ of $(p+1)/2$ which implies that every element $x$ of order $d$ belongs to one maximal subgroup $D_{p+1}$; so, it follows that $\Solv_G(x)=D_{p+1}$. Thus, the maximal subgroups $D_{p+1}$ can only be covered by the solvabilizers of involutions or elements of order a prime divisor of $(p+1)/2$. If the elements of order a prime divisor of $(p+1)/2$ are taken to cover the maximal subgroups $D_{p+1}$, $p(p-1)/2$ elements are needed. In the sequel, we will show that the solvabilizer of $p$ involutions can cover the maximal subgroups $D_{p+1}$ which is obviously less than $p(p-1)/2$. It is easy to see that each involution belongs to $(p-1)/2$ maximal subgroups $D_{p+1}$. By a similar argument in Theorem \ref{PSL(2, 2^f)}, we can partition the set of all maximal subgroups $D_{p+1}$ into $p$ disjoint subsets because the intersection of any two maximal subgroups $D_{p+1}$ can only contain an involution as a nontrivial element. So we can cover all maximal subgroups $D_{p+1}$ by the solvabilizer of $p$ involutions.

Assume next that $p \equiv 3 \mod 4$. All $p+1$ maximal subgroups of the form ${C_p}\rtimes C_{(p-1)/2}$ contain all $p+1$ abelian Sylow $p$-subgroups of $G$. Each of the $p^2-1$ elements $x$ of order $p$ belongs to one maximal subgroup ${C_p}\rtimes C_{(p-1)/2}$; thus, it follows that $\Solv_G(x)={C_p}\rtimes C_{(p-1)/2}$. Therefore, these maximal subgroups can only be covered by the solvabilizers of elements of order $p$ or a prime divisor of $(p-1)/2$. If we opt for elements of order $p$ to cover the maximal subgroups ${C_p}\rtimes C_{(p-1)/2}$, we require $p+1$ elements to cover all of them. In the subsequent discussion, we aim to cover them using $(p+1)/2$ elements of order prime divisors of $(p-1)/2$. Taking into account \cite[Theorem 2.1]{King}, it is evident that $G$ possesses a unique class of $p(p+1)/2$ conjugate cyclic groups of order $d$ for each divisor $d$ of $(p-1)/2$. Hence, in total, there are $p(p+1)(r-1)/2$ elements of order $r$ for $r$ as a prime divisor of $(p-1)/2$ in $G$. On the other hand, since each maximal subgroup ${C_p}\rtimes C_{(p-1)/2}$ has $p(r-1)$ elements of order $r$, all these subgroups contain $p(r-1)(p+1)$ elements of order $r$. It follows that any element of order $r$ belongs to two maximal subgroup ${C_p}\rtimes C_{(p-1)/2}$. Consequently, all these subgroups can be covered by $(p+1)/2$ elements of prime divisor order $(p-1)/2$. We will show that to cover the maximal subgroups $D_{p-1}$, at least $p-1$ involutions are needed. Given that the total number of involutions is $p(p-1)/2$ (refer to \cite[Theorem 2.1]{King}), it is evident that each involution belongs to $(p+1)/2$ maximal subgroups $D_{p-1}$. As $2$ does not divide $(q-1)/2$, it is easy to see that the intersection of any two maximal subgroups $D_{p-1}$ can only contain one involution as a nontrivial element. So we can partition the set of all maximal subgroups $D_{p-1}$ into $p$ disjoint subsets such that the maximal subgroups in a certain subsets intersect at one involution.
It is easily seen that any element $x$ of a prime divisor of $(p-1)/2$ order belongs to one maximal subgroup $D_{p-1}$ containing $x$. So $(p+1)/2$ elements of a prime divisor of $(p-1)/2$ order taken in previous step can cover $(p+1)/2$ maximal subgroups $D_{p-1}$. Therefore, to cover all $p(p+1)/2$ maximal subgroups $D_{p-1}$, we need to take at least $p-1$ involutions. This completes the proof.
\end{proof}
Further we will find the exact value of $\alpha(\PSL{2}{p})$ where $p \equiv 1 \mod 4$ in Theorem \ref{PSL(2, q)}.

\begin{theorem}\label{PSL(2, 3^p)}
Let $G=\PSL{2}{q}$ be a minimal simple group, where $q=3^p$ with $p$ as an odd prime. Then $\alpha(G)\geq (3q-1)/2$.
\end{theorem}
\begin{proof} 
The isomorphism classes of the maximal subgroups of $G$ are as follows: $D_{q-1}$, $D_{q+1}$, ${C_3}^p\rtimes C_{(q-1)/2}$, and $A_4$ (refer to \cite[Theorem 2.2]{Giudici}).

All subgroups ${C_3}^p\rtimes C_{(q-1)/2}$ contain all $q+1$ abelian Sylow $3$-subgroups of $G$. We should note that these maximal subgroups do not contain any involution. In fact, $q\equiv 3 \mod 4$. Hence, we can employ a similar argument to Theorem~\ref{PSL(2, q)} to demonstrate that, in order to cover $G$, a minimum of $(q+1)/2$ elements of prime divisor order $(q-1)/2$ and an additional $q-1$ involutions are necessary.
\end{proof}

\begin{theorem}
Let $G={\rm Sz}(2^p)$ be a minimal simple group, where $q=2^p$ and $p$ is an odd prime, Then $\alpha(G)\geq q^2+1$.
\end{theorem}
\begin{proof}
According to \cite[Theorem 4.1]{wilson}, the group ${\rm Sz}(q)$, where $q=2^p$, has an order of $q^2(q-1)t\cdot s$, with $t=q-\sqrt{2q}+1$ and $s=q+\sqrt{2q}+1$. Up to conjugacy, the maximal subgroups of ${\rm Sz}(q)$ include: a Sylow $2$-normalizer of order $q^2(q-1)$, normalizers of cyclic groups of orders $q-1$, $t$ and $s$, with structures $(C_2^p\cdot C_2^p)\rtimes C_{q-1}$, $D_{2(q-1)}$, $C_{s}\rtimes C_{4}$, and $C_{t}\rtimes C_{4}$ (see \cite{Bray}). $G$ has $q^2+1$ Sylow $2$-subgroups, resulting in $q^2+1$ maximal subgroups $(C_2^p\cdot C_2^p)\rtimes C_{q-1}$. Therefore, the number of involutions in $G$ is $(q^2+1)(q-1)$.

It is evident that the number of maximal subgroups isomorphic to $D_{2(q-1)}$ is  $|G:N_G(C_{q-1})|=|G:N_G(D_{2(q-1)})|=|G:D_{2(q-1)}|=q^2(q^2+1)/2$. Similarly, the number of maximal subgroups $C_t\rtimes C_{4}$ and $C_s\rtimes C_{4}$ can be observed to be $q^2(q-1)s/4$ and $q^2(q-1)t/4$, respectively. Moreover, as $G$ has $q^2+1$ Sylow $2$-subgroups, obviously the number of maximal subgroups $(C_2^p\cdot C_2^p)\rtimes C_{q-1}$ is $q^2+1$. If $x$ is an element of prime divisor order of $t$ or $s$, then $\Solv_G(x)=C_t\rtimes C_{4}$ or $\Solv_G(x)=C_s\rtimes C_{4}$, respectively. In the case where $x$ has a prime divisor order of $q-1$, then $\Solv_G(x)$ encompasses the maximal subgroups $(C_2^p\cdot C_2^p)\rtimes C_{q-1}$ and $D_{2(q-1)}$ containing $x$. Suppose $x$ is an involution; in that case, $\Solv_G(x)$ encompasses all the maximal subgroups of types $(C_2^p\cdot C_2^p)\rtimes C_{q-1}$, $D_{2(q-1)}$, $C_t\rtimes C_{4}$, and $C_s\rtimes C_{4}$ that contain $x$.

As a result of Corollary \ref{primes}, to cover the maximal subgroups $D_{2(q-1)}$, we can use the solvabilizer of either involutions or elements of order as prime divisors of $q-1$. 

In the subsequent analysis, we will establish that the minimal solvabilizer covering necessitates the inclusion of $q^2+1$ involutions to cover the $D_{2(q-1)}$ subgroups. We first note that $G$ has  $(q^2+1)(q-1)$ involutions and $\Phi(i)q^2(q^2+1)/2$ elements of order $i$ as a divisor of $q-1$ (See \cite{Alavi}). We can employ a similar argument as in Theorem \ref{PSL(2, 2^f)} to demonstrate that each element of order $i$ as a divisor of $q-1$ belongs to $\frac{\Phi(i)q^2(q^2+1)/2}{\Phi(i)q^2(q^2+1)/2}=1$ maximal subgroups $D_{2(q-1)}$ which implies that every element of order $r$ as a prime divisor of $q-1$ can only cover one maximal subgroup $D_{2(q-1)}$. Similarly, we can show that each involution belongs to $\frac{(q-1)q^2(q^2+1)/2}{(q^2+1)(q-1)}=q^2/2$ maximal subgroups $D_{2(q-1)}$. Note that the intersection of any two dihedral groups $D_{2(q-1)}$ can contain one involution as a nontrivial element because any element of order $i$ as a divisor of $q-1$ belongs to one maximal subgroups $D_{2(q-1)}$. So by a similar argument to Theorem \ref{PSL(2, 2^f)}, we can partition the set of all maximal dihedral subgroups $D_{2(q-1)}$ into $q^2+1$ disjoint subsets containing $q^2/2$ maximal dihedral subgroups $D_{2(q-1)}$ based on involution. If we pick the common involution in any subset, we can cover 
all maximal subgroups $D_{2(q-1)}$ by $q^2+1$ involutions.
\end{proof}
%____________________

\begin{theorem}\label{SolvabilizernumberGL} The general linear group $\GL{2}{q}$ for odd $q$ can be covered by a set of size $q$. Specifically, for $q \equiv 1 \mod 4$, the simple group $\PSL{2}{q}$ can be covered by a set of size $q$. In other words, $\alpha(\PSL{2}{q}) \leq q$ holds for $q \equiv 1 \mod 4$. \end{theorem}

\begin{proof} For every pair of distinct 1-dimensional subspaces $U$ and $W$, let $g_{U,W}$ represent the matrix with eigenpairs $(1,U)$ and $(-1,W)$.

Let $h$ be a two-by-two matrix. We distinguish between two cases: either $h$ has an eigenvalue over $\GF{q}$, or it does not.

    {\bf Case 1:}
    
If $h$ has an eigenvalue over $\GF{q}$, then it possesses some invariant 1-dimensional subspace $U$. Let $W$ be any other 1-dimensional subspace. With respect to any basis $\{\vec{u},\vec{w}\}$ with $\vec{u} \in U$ and $\vec{w} \in W$, all of $h$, $g_{U,W}$, and $g_{W,U}$ are upper triangular, and therefore, they are all contained in the Borel subgroup, a solvable group. Thus, $h$ is covered by every such $g_{U,W}$ and every such $g_{W,U}$, with the only restriction that either $U$ or $W$ contains an eigenvector of $h$.

    {\bf Case 2:}

If $h$ does not have eigenvalues over $\GF{q}$, there exists a $\GF{q}$-linear isomorphism $\phi:\GF{q^2} \to V$, where $V=\left(\GF{q}\right)^2$. In this case, $h = \phi(b)$ for some $b \in \GL{1}{q^2}$, and $\phi$ extends to $\GL{1}{q^2} \hookrightarrow \GL{2}{q}$. Consider any 1-dimensional subspace $U$, and let $\vec{u} \in U$. Choose $u \in \GF{q^2}$ such that $\phi(u) = \vec{u}$. Define the $\GF{q}$-linear map: $f: x \mapsto u^{1-q} x^q$. Using $\phi$, this can be considered as an element $A$ of $\GL{2}{q}$. Since $f^2 : x \mapsto u^{1-q^2} x^{q^2} = x$ is the identity, $A$ has eigenvalues $\pm1$. Since $f(u) = u^{1-q} u^q = u$, we have $(1,U)$ as an eigenpair of $A$, and thus $A = g_{U,W}$ for some $W$. As $h = \phi(b)$ for $b \in \GL{1}{q^2}$, the subgroup generated by $h$ and $A$ is isomorphic to the subgroup generated by $b$ and $f$, contained in the solvable group $\GammaL{1}{q^2}$. Hence, $h$ is covered by every such $g_{U,W}$, with the only restriction being the compatibility of the choice of $U$ and $W$.

Choose any subspace $U$, and let $\mathscr{X}$ consist of all $g_{U,W}$. As discussed in the two cases above, $h$ is covered by some element of $\mathscr{X}$—either because $U$ or some other subspace $W$ contains an eigenvector of $h$, or because $h$ lacks eigenvectors and $W$ was chosen to match the Frobenius-style map $f$ or matrix $A$. Hence, $\GL{2}{q}$ is covered by $q$ elements.

When $q \equiv 1 \mod 4$, the elements $g_{U,W}$ lie in the pre-image of $\PSL{2}{q}$ because their determinant, $-1$, is a square. According to Theorem \ref{quotient}, the solvabilizer coverings of $G$ and $G/N$ are identical when $N$ is the subgroup of scalar matrices.
\end{proof}

%\begin{theorem}\label{T:PSL(2,4t+1)}
%If $p$ is a prime such that $p \equiv 1 \mod 4$, then $ \alpha(\PSL{2}{p}) \ge p$.
%\end{theorem}
%\begin{proof}
%Let $G=\PSL{2}{p}$ where $p$ is a prime satisfying $p \equiv 1 \mod 4$. From Theorem $2.2$ in \cite{Giudici}, it is evident that the maximal subgroups of $\PSL{2}{p}$ (up to isomorphism) are: $D_{p-1}$, $D_{p+1}$, ${C_p}\rtimes C_{(p-1)/2}$, and either $A_4$, $A_5$, or $S_4$. Note that the solvabilizer of an element is the union of all maximal solvable subgroups containing the element. So we can use a similar argument to Theorem \ref{PSL(2, q)} to show that given an element $x$ of a prime divisor of $(p+1)/2$ order, $\Solv_G(x)=D_{p+1}$ where $D_{p+1}$ is the only maximal subgroup containing $x$. Now we can employ a similar argument to Theorem \ref{PSL(2, q)} to demonstrate that at least $p$ involutions are needed to cover all maximal subgroups $D_{p+1}$.
%\end{proof}

\begin{corollary}\label{PSL(2,4t+1)}
If $p$ is a prime such that $p \equiv 1 \mod 4$, then $\alpha(\PSL{2}{p})=p$.
    
\end{corollary}
\begin{proof}
Let $G=\PSL{2}{p}$ where $p$ is a prime satisfying $p \equiv 1 \mod 4$. In Theorem~\ref{PSL(2, q)} we demonstrated that $\alpha(G)\geq p$. It follows thus from Theorem \ref{SolvabilizernumberGL} that $\alpha(G)=p$.

\end{proof}

%________________________
\section{Possible Further Research Directions}
%_____________________________

In this section, we propose several conjectures that could pave the way for potential avenues of further research in the study of solvablizer number.

Drawing upon Theorem~\ref{PSL(2, 2^f)} and the outcomes from our computational analysis conducted with \gap, partially summarized in the table of Section~\ref{Appendix}, we posit the following conjecture:

\begin{conjecture}
If $G$ is a nonabelian simple group with $\alpha_{inv}(G)<\infty$, then $\alpha_{inv}(G)=\alpha(G)$
\end{conjecture}

Analyzing the table provided in section~\ref{Appendix}, we can put forth the following two conjectures, which serve as generalizations of Theorem~\ref{PSL(2, 2^f)}.

\begin{conjecture}
For $G=\PSL{2}{q}$, $q=2^f$, $\alpha(G)=q-1$.
\end{conjecture}
\begin{conjecture} $\alpha(\PSL{2}{q})=q$, where $q$ is a power of prime and  $q \equiv 1 \mod 4$. \end{conjecture}

Examining the table presented in section~\ref{Appendix}, we also suggest the following conjecture.

\begin{conjecture}\label{A5} If $\alpha(G)=3$, then there exists a composition factor of $G$ isomorphic to $A_5$. \end{conjecture}

\begin{remark}
    If Conjecture~\ref{A5} proves to be accurate, then $A_5$ would play a similar role in covering by solvabilizers as $C_2 \times C_2$ does in group covering by subgroups. However, notable distinctions exist; firstly, it is not necessary for $A_5$ to be a quotient group, it only needs to be a factor. For instance, as see from Theorem~\ref{TW} $\alpha(A_5\wr 2)=3$. Additionally, unlike in the subgroup covering case, it is not an if-and-only-if condition, as illustrated by the fact \cite{GAP4} that $\alpha(S_5)=5>3$, note that  $A_5$ is a composition factor for $S_5$.
\end{remark}

%____________________
%\appendix
\section{Computed table%: don't include yet
} \label{Appendix}

\newcommand\AGL[2]{\operatorname{AGL}_{#1}\left(#2\right)}
\newcommand\PGL[2]{\operatorname{PGL}_{#1}\left(#2\right)}
\newcommand\PSU[2]{\operatorname{PSU}_{#1}\left(#2^2/#2\right)}
\newcommand\PGU[2]{\operatorname{PGU}_{#1}\left(#2\right)}
\newcommand\PGammaL[2]{\operatorname{P\Gamma{}L}_{#1}\left(#2\right)}
\newcommand\PSigmaL[2]{\operatorname{P\Sigma{}L}_{#1}\left(#2\right)}
\newcommand\PSp[2]{\operatorname{PSp}_{#1}\left(#2\right)}
\newcommand\Alt[1]{\operatorname{A}_{#1}}
\newcommand\Sym[1]{\operatorname{S}_{#1}}
\newcommand\Mathieu[1]{\operatorname{M}_{#1}}
\newcommand\Janko[1]{\operatorname{J}_{#1}}
\newcommand\Sz[1]{\operatorname{Sz}\left(#1\right)}
\newcommand\Aut[1]{\operatorname{Aut}\left(#1\right)}
\newcommand\PrimitiveGroup[4]{\operatorname{Prim}(#1,#2)\cong#3}
\newcommand\PrimitiveGroupShort[4]{\operatorname{Prim}(#1,#2)}
\newcommand\StructureDescription[2]{#1}
\newcommand\WreathProduct[2]{#1 \wr #2}

%The solvablizer of an element $x$ of $G$ is the subset $$\Solv_G(x) = \{ y
%\in G : \langle x,y \rangle \text{ is solvable} \}.$$ If $R(G)$ is a normal
%solvable group of $G$, then $R(G)$ is contained in $\Solv_G(x)$ and $R(G) \Solv_G(x) /
%R(G) = \Solv_{G/R(G)}(xR(R))$. The solvablizer covering number of a group is the minimum
%size of a set $\mathcal{X}$ such that $G = \cup_{x\in\mathcal{X}} \Solv_G(x)$.

We analyze the solvabilizer covering number and the involutionary solvabilizer covering number of every Fitting-free group $G$ with an order less than $\mid M_{22}\mid$. The partial table below showcases our computed results, and we're open to sharing the full table with interested parties. Notably, for direct products, we utilize Theorems~\ref{Car} and \ref{Inv} to determine these covering numbers. % and an
%indication is made whether we can restrict to $\mathcal{X}$ consisting only
%of involutions. 
\cite{GAP4} is used to construct the maximal solvable subgroups and
thus the solvablizer graph. Then a mixed integer linear programming solver
is used to find coverings of minimal size. %A separate column is included
%for the restricted problem of only using solvablizers of involutions,
%which is often quite effective, for example for $\PGL{2}{q}$.

Certain entries are represented as intervals: $[a, b]$ signifies the discovery of a cover using $b$ non-identity elements (or involutions in that column), while it has been proven that no cover employing fewer than $a$ non-identity elements (or involutions) is feasible. The notation $\infty$ is used to indicate the absence of a cover, a circumstance exclusive to the involution column.

%Don't include in the paper directly, but hopefully this helps with
%conjectures.

There is a unique group $G$ such that $S \times T \leq G \leq A \times B$,
$[A:S] = [B:T] = [G:S\times T] = 2$, and $G \neq S \times B, A \times T$.
This group is denoted $A \Yup B$, indicating a direct product, except where
the top is squished. This is only used when $S,T$ are simple, so they
are identified as the socle of $A$ and $B$. In the table, it's important to note that there exists a variety of nonsolvable groups, and due to the extensive range of these groups, not all notations and classifications are exhaustively explained. Other group names are as found in
GAP, for example in the output of StructureDescription. 

We extend our gratitude to Dr. Yuan Zhou at the University of Kentucky for assistance with the integer programming problems.
\vspace{1cm}

\( \begin{array}{rcrc}
\text{Order} & \text{Name} & \alpha & \alpha_{inv} \\ \hline
60 & A_5\cong\PSL{2}{4} & 3 & 3 \\
120 & S_5\cong\PGL{2}{5} & 5 & 5 \\
168 & \PSL{2}{7} & 5 & \infty \\
336 & \PGL{2}{7} & 7 & 7 \\
360 & A_6\cong\PSL{2}{9} & 9 & 9 \\
504 & \PSL{2}{8} & 7 & 7 \\
660 & \PSL{2}{11} & 15 & \infty \\
720 & \Mathieu{10} & 9 & 9 \\
720 & \PGL{2}{9} & 8 & 8 \\
720 & \Sym{6} & 9 & 9 \\
1092 & \PSL{2}{13} & 13 & 13 \\
1320 & \PGL{2}{11} & 11 & 11 \\
1440 & \PGammaL{2}{9} & 9 & 9 \\
1512 & \PGammaL{2}{8} & 7 & 7 \\
2184 & \PGL{2}{13} & 13 & 13 \\
2448 & \PSL{2}{17} & 17 & 17 \\
2520 & \Alt{7} & 41 & \infty \\
3420 & \PSL{2}{19} & 31 & \infty \\
%3600 & \PSL{2}{4} \times \PSL{2}{4} & 3 & 3 \\
4080 & \PSL{2}{16} & 15 & 15 \\
%4896 & \PGL{2}{17} & 17 & 17 \\
5040 & \Sym{7} & 21 & 21 \\

5616 & \PSL{3}{3} & 25 & \infty \\
6048 & \PSU{3}{3} & 49 & \infty \\

6072 & \PSL{2}{23} & [38,41] & \infty \\

%6840 & \PGL{2}{19} & 19 & 19 \\
%7200 & \PSL{2}{4} \times \PGL{2}{5} & 3 & 3 \\
%7200 & \PrimitiveGroup{60}{2}{\Alt{5}^2.2}{A5 : S5} & 5 & 5 \\
7200 & \WreathProduct{\PSL{2}{4}}{2} & 3 & 3 \\

7800 & \PSL{2}{25} & 25 & 25 \\

7920 & \Mathieu{11} & 45 & \infty \\

8160 & \PSL{2}{16}:2 & 17 & 17 \\
9828 & \PSL{2}{27} & [32,48] & [33,49] \\
%10080 & \PSL{2}{7} \times \PSL{2}{4} & 3 & 3 \\

11232 & \Aut{\PSL{3}{3}} & 21 & 21 \\

%12096 & \Aut{\PSU{3}{3}} & 36 & 36 \\
%12144 & \PGL{2}{23} & 23 & 23 \\
12180 & \PSL{2}{29} & 29 & 29 \\

%14400 & \PGL{2}{5} \times \PGL{2}{5} & 5 & 5 \\
%14400 & (\PSL{2}{4} \times \PSL{2}{4}) : (C2 \times C2) & 5 & 5 \\
%14400 & (\PSL{2}{4} \times \PSL{2}{4}) : C4 & 5 & 5 \\
14880 & \PSL{2}{31} & [48,60] & \infty \\

%15600 & \PGL{2}{25} & 25 & 25 \\
%15600 & \PSigmaL{2}{25} & 25 & 25 \\
15600 & \PSL{2}{25}.2 & 25 & 25 \\
%16320 & \PGammaL{2}{16} & 17 & 17 \\
%19656 & \PGL{2}{27} & 26 & 26 \\
20160 & A_8\cong\PSL{4}{2} & 54 & 54 \\ % g_20160_1
20160 & \PSL{3}{4} & 169 & \infty \\ % g_20160_2
%20160 & \PGL{2}{7} \times \PSL{2}{4} & 3 & 3 \\ % g_20160_3
%20160 & \PSL{2}{7} \times \PGL{2}{5} & 5 & 5 \\ % g_20160_4
%20160 & \PSL{2}{7} : \PGL{2}{5} & 5 & 5 \\ % \TransitiveGroup{40}{11251} g_20160_5
%21600 & \PSL{2}{9} \times \PSL{2}{4} & 3 & 3 \\
%24360 & \PGL{2}{29} & 29 & 29 \\
25308 & \PSL{2}{37} & 37 & 37 \\
25920 & \PSp{4}{3} & 9 & 9 \\

%28224 & \PSL{2}{7} \times \PSL{2}{7} & 5 & \infty \\
%28800 & \WreathProduct{\PGL{2}{5}}{2} & 5 & 5 \\
29120 & \Sz{8} & [106,155] & [106,155] \\

%29484 & \PSigmaL{2}{27} & [32,42] & \infty \\
%29760 & \PGL{2}{31} & 31 & 31 \\
%30240 & \PSL{2}{8} \times \PSL{2}{4} & 3 & 3 \\
31200 & \PGammaL{2}{25} & 25 & 25 \\
32736 & \PSL{2}{32} & 31 & 31 \\
34440 & \PSL{2}{41} & 41 & 41 \\
%39600 & \PSL{2}{11} \times \PSL{2}{4} & 3 & 3 \\
39732 & \PSL{2}{43} & [65,93] & \infty \\
40320 & \Sym{8} & [23,24] & [23,24] \\ % g_40320_1

%40320 & \PSigmaL{3}{4} & 129 & 129 \\ % g_40320_3
%40320 & \PrimitiveGroup{56}{3}{PSL(3, 4).2_1}{PSL(3,4) : C2} & [63,96] & [64,96] \\ % g_40320_2
%40320 & \PrimitiveGroup{56}{4}{PSL(3, 4).2_2}{PSL(3,4) : C2} & [48,58] & [50,58] \\ % g_40320_4
%40320 & \PGL{2}{7} \times \PGL{2}{5} & 5 & 5 \\ % g_40320_5
%43200 & \Mathieu{10} \Yup \PGL{2}{5} & 5 & 5 \\ % g_43200_5
%43200 & \Mathieu{10} \times \PSL{2}{4} & 3 & 3 \\ % g_43200_1
%43200 & \PGL{2}{9} \Yup \PGL{2}{5} & 5 & 5 \\ % g_43200_7
%43200 & \PGL{2}{9} \times \PSL{2}{4} & 3 & 3 \\ % g_43200_3
%43200 & \PSL{2}{9} \times \PGL{2}{5} & 5 & 5 \\ % g_43200_4
%43200 & \PSigmaL{2}{9} \Yup \PGL{2}{5} & 5 & 5 \\ % g_43200_6
%43200 & \PSigmaL{2}{9} \times \PSL{2}{4} & 3 & 3 \\ % g_43200_2
\end{array} \)

\(
\begin{array}{rcrc}
\text{Order} & \text{Name} & \alpha & \alpha_{inv} \\ \hline
50616 & \PGL{2}{37} & 37 & 37 \\
%51840 & \Aut{PSp{4}{3}} & 36 & 36 \\
51888 & \PSL{2}{47} & [70,107] & \infty \\
%56448 & \PSL{2}{7} \times \PGL{2}{7} & 5 & 7 \\ % g_56448_1

56448 & \PSL{2}{7} : \PGL{2}{7} & 9 & 22 \\ % g_56448_2
%56448 & \PSL{2}{7} \wr 2 & 5 & 14 \\ % g_56448_3
58800 & \PSL{2}{49} & 49 & 49 \\
%58968 & \PGammaL{2}{27} & 27 & 27 \\
%60480 & \PGL{3}{4} & [70,112] & \infty \\
%60480 & \PSL{2}{8} \times \PGL{2}{5} & 5 & 5 \\
%60480 & \PSL{2}{9} \times \PSL{2}{7} & 5 & 9 \\

62400 & \PSU{3}{4} & 160 & \infty \\
%65520 & \PSL{2}{13} \times \PSL{2}{4} & 3 & 3 \\
%68880 & \PGL{2}{41} & 41 & 41 \\
74412 & \PSL{2}{53} & 53 & 53 \\

%79200 & \PGL{2}{11} \Yup \PGL{2}{5} & 5 & 5 \\ % g_79200_3
%79200 & \PGL{2}{11} \times \PSL{2}{4} & 3 & 3 \\ % g_79200_1
%79200 & \PSL{2}{11} \times \PGL{2}{5} & 5 & 5 \\ % g_79200_2
%79464 & \PGL{2}{43} & 43 & 43 \\
%80640 & \PSL{3}{4} : (2 \times 2) & [58,96] & [62,96] \\
%84672 & \PSL{2}{7} \times \PSL{2}{8} & 5 & 7 \\
%86400 & \Mathieu{10} \times \PGL{2}{5} & 5 & 5 \\
%86400 & \PGL{2}{9} \times \PGL{2}{5} & 5 & 5 \\
%86400 & \Sym{6} \times \PGL{2}{5} & 5 & 5 \\
%86400 & ((A6 : C2) : C2) \times \PSL{2}{4} & 3 & 3 \\
86400 & A5 : ((A6 : C2) : C2) & 5 & 5 \\
%86400 & A5 : ((A6 : C2) : C2) & 5 & 5 \\
%86400 & A5 : ((A6 . C2) : C2) & 5 & 5 \\
%87360 & \Aut{\Sz{8}} & [97,162] & [101,162] \\
%90720 & \PGammaL{2}{8} \times \PSL{2}{4} & 3 & 3 \\
95040 & \Mathieu{12} & [277,327] & \infty \\
102660 & \PSL{2}{59} & [86,143] & \infty \\
103776 & \PGL{2}{47} & 47 & 47 \\
%110880 & \PSL{2}{7} \times \PSL{2}{11} & 5 & \infty \\

%112896 & \PGL{2}{7} \times \PGL{2}{7} & 7 & 7 \\
%112896 & \PrimitiveGroup{64}{70}{\PSL{2}{7}^2.2^2}{(PSL(3,2) x PSL(3,2)) : (C2 x C2)} & 9 & [16,18] \\
%112896 & \PrimitiveGroup{64}{71}{\PSL{2}{7}^2.4}{(PSL(3,2) x PSL(3,2)) : C4} & 9 & [21,22] \\
113460 & \PSL{2}{61} & 61 & 61 \\
117600 & \PGL{2}{49} & 49 & 49 \\
117600 & \PSigmaL{2}{49} & 49 & 49 \\
%117600 & \PrimitiveGroup{50}{6}{\PSL{2}{49}.2_3}{PSL(2,49) . C2} & 49 & 49 \\
%120960 & \Mathieu{10} \Yup \PGL{2}{7} & 9 & 9 \\
%120960 & \Mathieu{10} \times \PSL{2}{7} & 5 & 9 \\
120960 & \PGL{2}{9} \Yup \PGL{2}{7} & 9 & 9 \\
%120960 & \PGL{2}{9} \times \PSL{2}{7} & 5 & 8 \\
%120960 & \PGammaL{3}{4} & [70,118] & [126,129] \\
%120960 & \PSL{2}{9} \times \PGL{2}{7} & 7 & 7 \\
%120960 & \PSigmaL{2}{9} \Yup \PGL{2}{7} & 9 & 9 \\
%120960 & \PSigmaL{2}{9} \times \PSL{2}{7} & 5 & 9 \\
%120960 & \PrimitiveGroup{105}{4}{\PSL{3}{4}.\Sym{3}}{PSL(3,4) : S3} & [57,72] & [58,76] \\
%120960 & \PrimitiveGroup{105}{5}{\PSL{3}{4}.6}{PSL(3,4) : C6} & [47,64] & [49,60] \\

%124800 & \PrimitiveGroup{65}{4}{\PSU{3}{4}.2}{PSU(3,4) : C2} & 80 & 80 \\
126000 & \PSU{3}{5} & [272,429] & \infty \\
%129600 & \PSL{2}{9} \times \PSL{2}{9} & 9 & 9 \\
131040 & \PGL{2}{13} \Yup \PGL{2}{5} & 5 & 5 \\
%131040 & \PGL{2}{13} \times \PSL{2}{4} & 3 & 3 \\
%131040 & \PSL{2}{13} \times \PGL{2}{5} & 5 & 5 \\
%146880 & \PSL{2}{17} \times \PSL{2}{4} & 3 & 3 \\
%148824 & \PGL{2}{53} & 53 & 53 \\
150348 & \PSL{2}{67} & [98,165] & \infty \\
%151200 & \Alt{7} \times \PSL{2}{4} & 3 & 3 \\
%158400 & \PGL{2}{11} \times \PGL{2}{5} & 5 & 5 \\
%163680 & \PGammaL{2}{32} & 31 & 31 \\
%169344 & \PGL{2}{7} \times \PSL{2}{8} & 7 & 7 \\
%172800 & \PGammaL{2}{9} \times \PGL{2}{5} & 5 & 5 \\
175560 & \Janko{1} & [133,249] & [133,249] \\
178920 & \PSL{2}{71} & [104,181] & \infty \\
181440 & \Alt{9} & [40,95] & [56,110] \\
%181440 & \PGammaL{2}{8} \times \PGL{2}{5} & 5 & 5 \\
%181440 & \PSL{2}{9} \times \PSL{2}{8} & 7 & 7 \\
%183456 & \PSL{2}{7} \times \PSL{2}{13} & 5 & 13 \\
190080 & \Mathieu{12} & [74,161] & [74,161] \\
194472 & \PSL{2}{73} & 73 & 73 \\
%205200 & \PSL{2}{19} \times \PSL{2}{4} & 3 & 3 \\
205320 & \PGL{2}{59} & 59 & 59 \\
%216000 & \PSL{2}{4} \times \PSL{2}{4} \times \PSL{2}{4} & 3 & 3 \\
%221760 & \PGL{2}{11} \Yup \PGL{2}{7} & 9 & [29,38] \\
%221760 & \PGL{2}{7} \times \PSL{2}{11} & 7 & 7 \\
%221760 & \PSL{2}{7} \times \PGL{2}{11} & 5 & 11 \\
%225792 & \WreathProduct{\PGL{2}{7}}{2} & 7 & 7 \\
%226920 & \PGL{2}{61} & [60,61] & 61 \\
%235200 & \PGammaL{2}{49} & 49 & 49 \\
%237600 & \PSL{2}{9} \times \PSL{2}{11} & 9 & 9 \\
%241920 & \Aut{\PSL{3}{4}} & [56,70] & [57,70] \\
%241920 & \Mathieu{10} \times \PGL{2}{7} & 7 & 7 \\
%241920 & \PGL{2}{9} \times \PGL{2}{7} & 7 & 7 \\
%241920 & \PGammaL{2}{9} \times \PSL{2}{7} & 5 & 9 \\
%241920 & \PSigmaL{2}{9} \times \PGL{2}{7} & 7 & 7 \\
%241920 & \StructureDescription{\PSL{3}{2} : ((A6 . C2) : C2)}{\PSL{2}{9} \times \PSL{2}{7}} & 9 & 9 \\
%241920 & \StructureDescription{\PSL{3}{2} : ((A6 : C2) : C2)}{\PSL{2}{9} \times \PSL{2}{7}} & 9 & 9 \\
%241920 & \StructureDescription{\PSL{3}{2} : ((A6 : C2) : C2)}{\PSL{2}{9} \times \PSL{2}{7}} & 9 & 9 \\
%244800 & \PSL{2}{16} \times \PSL{2}{4} & 3 & 3 \\
246480 & \PSL{2}{79} & [116,204] & \infty \\
%249600 & \Aut{\PSU{3}{4}} & 80 & 80 \\
%252000 & \Aut{\PSU{3}{5}} & [153,168] & [156,168] \\
%254016 & \PSL{2}{7} \times \PGammaL{2}{8} & 5 & 7 \\
%254016 & \PSL{2}{8} \times \PSL{2}{8} & 7 & 7 \\
%259200 & \Mathieu{10} \Yup \PGL{2}{9} & 9 & 9 \\
%259200 & \Mathieu{10} \Yup \PSigmaL{2}{9} & 9 & 9 \\
%259200 & \Mathieu{10} \times \PSL{2}{9} & 9 & 9 \\
%259200 & \PGL{2}{9} \Yup \PSigmaL{2}{9} & 9 & 9 \\
%259200 & \PGL{2}{9} \times \PSL{2}{9} & 8 & 8 \\
%259200 & \PSigmaL{2}{9} \times \PSL{2}{9} & 9 & 9 \\
%259200 & \PrimitiveGroup{360}{2}{Alt(6)^2.2}{A6 : S6} & 9 & 9 \\
%259200 & \PrimitiveGroup{360}{3}{Alt(6)^2.2}{A6 : (A6 : C2)} & 9 & 9 \\
%259200 & \PrimitiveGroup{360}{4}{Alt(6)^2.2}{A6 : (A6 . C2)} & 9 & 9 \\
%259200 & \WreathProduct{\PSL{2}{9}}{2} & 9 & 9 \\
%262080 & \PGL{2}{13} \times \PGL{2}{5} & 5 & 5 \\
262080 & \PSL{2}{64} & 63 & 63 \\
265680 & \PSL{2}{81} & 81 & 81 \\
285852 & \PSL{2}{83} & [122,218] & \infty \\
%293760 & \PGL{2}{17} \Yup \PGL{2}{5} & 5 & 5 \\
%293760 & \PGL{2}{17} \times \PSL{2}{4} & 3 & 3 \\
%293760 & \PSL{2}{17} \times \PGL{2}{5} & 5 & 5 \\
300696 & \PGL{2}{67} & 67 & 67 \\
%302400 & \Alt{7} \times \PGL{2}{5} & 5 & 5 \\
302400 & \Sym{7} \Yup \PGL{2}{5} & 5 & 5 \\
%302400 & \Sym{7} \times \PSL{2}{4} & 3 & 3 \\
%332640 & \PSL{2}{8} \times \PSL{2}{11} & 7 & 7 \\
%336960 & \PSL{3}{3} \times \PSL{2}{4} & 3 & 3 \\
352440 & \PSL{2}{89} & 89 & 89 \\
357840 & \PGL{2}{71} & [70,71] & [70,71] \\
%362880 & \Mathieu{10} \times \PSL{2}{8} & 7 & 7 \\
%362880 & \PGL{2}{9} \times \PSL{2}{8} & 7 & 7 \\
%362880 & \PSU{3}{3} \times \PSL{2}{4} & 3 & 3 \\
%362880 & \PSigmaL{2}{9} \times \PSL{2}{8} & 7 & 7 \\

362880 & \Sym{9} & 36 & 36 \\
%364320 & \PSL{2}{23} \times \PSL{2}{4} & 3 & 3 \\
%366912 & \PGL{2}{13} \Yup \PGL{2}{7} & 9 & 13 \\
%366912 & \PGL{2}{7} \times \PSL{2}{13} & 7 & 7 \\
%366912 & \PSL{2}{7} \times \PGL{2}{13} & 5 & 13 \\

372000 & \PSL{3}{5} & [261,266] & \infty \\
378000 & \PGU{3}{5} & [271,465] & \infty \\
388944 & \PGL{2}{73} & 73 & 73 \\
%393120 & \PSL{2}{9} \times \PSL{2}{13} & 9 & 9 \\

%410400 & \PGL{2}{19} \Yup \PGL{2}{5} & 5 & 5 \\
%410400 & \PGL{2}{19} \times \PSL{2}{4} & 3 & 3 \\
%410400 & \PSL{2}{19} \times \PGL{2}{5} & 5 & 5 \\
%411264 & \PSL{2}{7} \times \PSL{2}{17} & 5 & 17 \\
%423360 & \PSL{2}{7} \times \Alt{7} & 5 & \infty \\
%432000_1 & \PSL{2}{4} \times \PGammaL{2}{4} \times \PSL{2}{4} & 3 & 3 \\
%432000_2 & \PrimitiveGroup{60}{2}{Alt(5)^2.2}{A5 : S5} \times \PSL{2}{4} & 3 & 3 \\
%432000 & \PrimitiveGroupShort{60}{2}{Alt(5)^2.2}{A5 : S5} \Yup \PrimitiveGroupShort{60}{2}{Alt(5)^2.2}{A5 : S5} \Yup \PrimitiveGroupShort{60}{2}{Alt(5)^2.2}{A5 : S5} & 5 & 5 \\
%432000 & \WreathProduct{\PSL{2}{4}}{2} \Yup \PGammaL{2}{4} & 3 & 3 \\
%432000_5 & \WreathProduct{\PSL{2}{4}}{2} \times \PSL{2}{4} & 3 & 3 \\
%435600 & \PSL{2}{11} \times \PSL{2}{11} & 15 & \infty \\
443520 & \Mathieu{22} & [2181,2501] & \infty \\
%604800 & J_2 & [361,828] & [361,830] \\
\end{array} \)

\vspace{1cm}

%____________________

%_____________________

\end{document}